\documentclass[12pt]{amsart}
\usepackage[utf8]{inputenc}
\usepackage {amssymb}
\usepackage {amsmath}
\usepackage{amsthm}
\usepackage{graphicx}
\usepackage {amscd}
\usepackage {epic}
\usepackage{xypic}
\usepackage[colorlinks, citecolor=blue]{hyperref}
\usepackage[left=1.5in,right=1.5in,top=1.5in,bottom=1.5in]{geometry}
\usepackage{soul}
\usepackage{comment}
\usepackage{enumitem}
\usepackage{mathrsfs}

\theoremstyle{plain}
\newtheorem{theorem}{Theorem}[section]
\newtheorem{corollary}[theorem]{Corollary}
\newtheorem{lemma}[theorem]{Lemma}

\newtheorem{proposition}[theorem]{Proposition}
\newtheorem{definition-lemma}[theorem]{Definition-Lemma}
\theoremstyle{remark}
\newtheorem{remark}[theorem]{Remark}

\theoremstyle{definition}

\newcommand{\Pic}[0]{\operatorname{Pic}}

\def\>{\geq}
\newcommand{\mbQ}{\mathbb{Q}}

\def\mcO{\mathcal{O}}
\newcommand{\num}{\equiv}

\newcommand{\OO}{{\mathcal{O}}}
\newcommand{\Q}{{\mathbb{Q}}}

\newcommand{\mult}{{\rm mult}}
\newcommand{\Supp}{{\rm Supp}}

\newcommand{\mbC}{\mathbb{C}}

\newcommand{\bir}{\dashrightarrow}

\def\lrd{\lfloor}
\def\rrd{\rfloor}

\def\>{\geq}

\def\mcO{\mathcal{O}}

\def\mcC{\mathcal{C}}

\def\mcE{\mathcal{E}}

\def\mcP{\mathcal{P}}

\def\msF{\mathscr{F}}
\def\msG{\mathscr{G}}

\def\bfR{\mathbf{R}}

\def\lrd{\lfloor}
\def\rrd{\rfloor}

\def\Ex{\operatorname{Ex}}
\def\dim{\operatorname{dim}}

\def\Alb{\operatorname{Alb}}
\def\Pic{\operatorname{Pic}}
\def\Proj{\operatorname{Proj}}

\theoremstyle{definition}
\newtheorem{definition}[theorem]{Definition}
\theoremstyle{definition}

\numberwithin{equation}{section}

\theoremstyle{remark}
\newtheorem{claim}[theorem]{Claim}

\title[Good Minimal Models for K\"ahler Varieties]{Existence of Good Minimal Models for K\"ahler Varieties of Maximal Albanese Dimension}

\author{Omprokash Das}
\address{School of Mathematics\\
Tata Institute of Fundamental Research\\
Homi Bhabha Road, Navy Nagar\\
Colaba, Mumbai 400005}
\email{omdas@math.tifr.res.in}
\email{omprokash@gmail.com}

\author{Christopher Hacon}
\address{Department of Mathematics\\
University of Utah\\
155 S 1400 E\\
Salt Lake City, Utah 84112}
\email{hacon@math.utah.edu}
\date{}

\begin{document}

\maketitle

\begin{abstract}
    In this short article we show that if $(X, B)$ is a compact K\"ahler klt pair of maximal Albanese dimension, then it has a good minimal model, i.e. there is a bimeromorphic contraction $\phi:X\bir X'$ such that $K_{X'}+B'$ is semi-ample.
\end{abstract}

\section{Introduction}
 
The main result of this article is the following
  \begin{theorem}\label{thm:main}
 Let $(X,B)$ be a compact K\"ahler klt pair of maximal Albanese dimension. Then $(X,B)$ has a good minimal model.
 \end{theorem}
 
 This generalizes the main result of \cite {Fuj15} from the projective case to the K\"ahler case. The main idea is to observe that replacing $X$ by an appropriate resolution, then the Albanese morphism $X\to A$ is projective and so by \cite{DHP22} and \cite{Fuj22} we may run the relative MMP over $A$. Thus we may assume that $K_X+B$ is nef over $A$. If $X$ is projective and $K_X+B$ is not nef, then by the cone theorem, $X$ must contain a $K_X+B$ negative rational curve $C$. Since $A$ contains no rational curves, then $C$ is vertical over $A$, contradicting the fact that $K_X+B$ is nef over $A$ \cite{Fuj15}.
 Unluckily, the cone theorem is not known for K\"ahler varieties and so we pursue an different argument. It would be interesting to find an alternative proof based on the arguments of \cite{CH20}.

\section{Preliminaries}
An \textit{analytic variety} or simply a \textit{variety} is a reduced irreducible complex space.  
Let $X$ be a compact K\"ahler manifold and $\Alb(X)$ is the \textit{Albanese torus} (not necessarily an Abelian variety). Then by $a:X\to \Alb(X)$ we will denote the \textit{Albanese morphism}. This morphism can also be characterized via the following universal property: $a:X\to \Alb(X)$ is the Albanese morphism if for every morphism $b:X\to T$ to a complex torus $T$ there is a unique morphism $\phi:\Alb(X)\to T$ such that $b=\phi\circ a$.\\
The Albanese dimension of $X$ is defined as $\dim a(X)$. We say that $X$ has maximal Albanese dimension if $\dim a(X)=\dim X$ or equivalently, the Albanese morphism $a:X\to \Alb(X)$ is \textit{generically finite} onto its image. For the definition of \textit{singular} K\"ahler space see \cite{HP16} or \cite{DH20}.\\
A compact analytic variety $X$ is said to be in  \textit{Fujiki's class} $\mcC$ if $X$ is bimeromorphic to a compact K\"ahler manifold $Y$. In particular, there is a resolution of singularities $f:Y\to X$ such that $Y$ is a compact K\"ahler manifold.\\

\begin{definition}\label{def:albanese-morphism}
Let $X$ be a compact analytic variety in Fujiki's class $\mcC$. Assume that $X$ has rational singularities. Choose a resolution of singularities $\mu:Y\to X$ such that $Y$ is a K\"ahler manifold and let $a_Y:Y\to \Alb(Y)$ be the Albanese morphism of $Y$. Then from the proof of \cite[Lemma 8.1]{Kaw85} it follows that $a_Y\circ \mu^{-1}:X\bir \Alb(Y)$ extends to a unique morphism $a:X\to \Alb(X):=\Alb(Y)$. We call this morphism the Albanese morphism of $X$. Observe that $a:X\to \Alb(X)$ satisfies the universal property stated above. The Albanese dimension of $X$ is defined as above. Note that if $X$ is a compact analytic variety  with rational singularities, bimeromorphic to a complex torus $A$, then $A\cong\Alb(X)$ and $X\to A$ is a bimeromorphic morphism.
\end{definition}

The following result is well known, however, for a lack of proper reference and convenience of the readers we give a complete proof here.
\begin{lemma}\label{lem:effective-k-dim}
Let $A$ be a complex torus and $X\subset A$ is an analytic subvariety. Then for any resolution of singularities $\mu:Y\to X$, $H^0(Y, \omega_Y)\neq \{0\}$. 
\end{lemma}

\begin{proof}
Let $\mu:Y\to X$ be a resolution of singularities of $X$. Choose a point $p\in X\setminus (X_\textsubscript{sing}\cup \mu(\Ex(\mu)))$ and let $(x_1, x_2,\ldots, x_n)$ be a local coordinate on $A$ near $p$, where $n=\dim A$. Since $p\in X$ is a smooth point, $X$ is local complete intersection near $p$. Thus we may assume that $X=\mbox{Zero}(x_{d+1},\ldots, x_n)$ near $p$, where $d=\dim X$. Then $dx_1\wedge dx_2\wedge\cdots\wedge dx_d$ is a local section of $\Omega_A^d$ near $p$ whose restriction to $X$ is a non-zero local section of $\Omega_X^d$ near $p\in X$. Since $\Omega_A^d$ is a trivial vector bundle, there is a non-zero global section $\Theta\in H^0(A, \Omega_A^d)$ which restricts to $dx_1\wedge dx_2\wedge\cdots\wedge dx_d$ near $p$, and hence $\mu^*\Theta|_X$ is a non-zero global section of $\Omega_Y^d:=\omega_Y$. In particular, $H^0(Y, \omega_Y)\neq \{0\}$.
\end{proof}

\begin{corollary}\label{cor:maximal-albanese-kappa}
Let $X$ be a compact analytic variety in Fujiki's class $\mcC$ with canonical singularities. If $X$ has maximal Albanese dimension, then $\kappa(X)\>0$.
\end{corollary}

\begin{proof}
First note that if $f:W\to X$ is a proper bimeromorphic morphism, then $\kappa(X)\>0$ if and only if $\kappa(W)\>0$, since $X$ has canonical singularities. Now let $a:X\to \Alb(X)$ be the Albanese morphism, $Y:=a(X)$, and $\pi:Z\to Y$ is a resolution of singularities of $Y$. Then $\kappa(Z)\>0$ by Lemma \ref{lem:effective-k-dim}. Note that there is a generically finite meromorphic map $\phi:X\bir Z$; resolving the graph of $\phi$ we may assume that $X$ is smooth and $\phi:X\to Z$ is a morphism. Then $K_X=\phi^*K_Z+E$, where $E\>0$ is an effective divisor. Therefore $\kappa(X)\>0$, since $\kappa(Z)\>0$.  
\end{proof}

\subsection{Fourier-Mukai transform} Let $T$ be a complex torus of dimension $g$ and $\hat T=\Pic^0(T)$ its dual torus. Let $p_T: T\times\hat T\to T$ and $p_{\hat T}:T\times\hat T\to \hat T$ be the projections, and $\mcP$ the normalized Poincar\'e line bundle on $T\times\hat T$ so that $\mcP|_{T\times\{0\}}\cong \mcO_T$ and $\mcP|_{\{0\}\times \hat T}\cong\mcO_{\hat T}$. Let $\hat S$ be the functor from the category of $\mcO_T$-sheaves to the category of $\mcO_{\hat T}$-sheaves, defined by 
\[
\hat S(\msF):=p_{\hat T, *}(p_T^*\msF\otimes\mcP),
\]
where $\msF$ is a sheaf of $\mcO_T$-modules.
Similarly, $S$ is a functor from the category of $\mcO_{\hat T}$-sheaves to the category of $\mcO_T$-sheaves, defined as
\[
S(\msG):=p_{T, *}(p_{\hat T}^*\msG\otimes\mcP),
\]
where $\msG$ is a sheaf of $\mcO_{\hat T}$-modules.\\
The corresponding derived functors are 
\[
\bfR\hat S(\boldsymbol{\cdot}):=\bfR p_{\hat T, *}(p_T^*(\boldsymbol{\cdot})\otimes\mcP) \mbox{ and } \bfR S(\boldsymbol{\cdot} ):=\bfR p_{T, *}(p_{\hat T}^*(\boldsymbol{\cdot})\otimes\mcP).
\]
Recall the following fundamental result of Mukai \cite[Theorem 2.2, and (3.8)]{Mu81}, \cite[Theorem 13.1]{PPS17}
\begin{theorem}\label{thm:fm-isomorphism}
With notations and hypothesis as above, there are isomorphisms of functors (on the bounded derived category of coherent sheaves)
\begin{equation*}
        \bfR\hat S\circ \bfR S \cong (-1)_{\hat T}^*[-g], \qquad
        \bfR S\circ\bfR\hat S \cong (-1)_T^*[-g],\end{equation*}
\begin{equation*}
        \mathbf\Delta_T \circ \bfR S=((-1_T)^*\circ \bfR S \circ \mathbf\Delta _{\hat T})[-g].
\end{equation*}
\end{theorem}~\\
Recall that $\mathbf\Delta_T(\boldsymbol{\cdot}):=\mathbf R{\mathcal Hom}(\boldsymbol{\cdot},\mathcal O _T)[g]$ is the dualizing functor.

\begin{definition}\label{def:h-vector-bundle}
Let $A$ be a complex torus. For $a\in A$, let $t_a:A\to A$ be the usual translation morphism defined by $a$. A vector bundle $\mcE$ on $A$ is called \textit{homogeneous}, if $t_a^*\mcE\cong\mcE$ for all $a\in A$. 
\end{definition}

\begin{remark}\label{rmk:h-vector-bundle}
Let $A$ be a complex torus, $\hat A$ the dual torus and $\dim A=\dim \hat A=g$. Then from the proof of \cite[Example 3.2]{Mu81} it follows that $R^g\hat S$ gives an equivalence of categories \[\mathbf H_A:=\{\mbox{Homogeneous vector bundles on } A\},\qquad {\rm and}\] \[\mathbf C^f_{\hat A}:=\{\mbox{Coherent sheaves on } \hat A \mbox{ supported at finitely many points} \}.\] Note that in \cite{Mu81} the results are all stated for abelian varieties, however, we observe that in the proof of \cite[Example 3.2]{Mu81} the main arguments follow from Theorem \ref{thm:fm-isomorphism} and the isomorphisms in \cite[(3.1), page 158]{Mu81}, both of which hold over complex tori. In particular, \cite[Example 3.2]{Mu81} holds for complex tori.  
\end{remark}

We will need the following result on the rational singularity of (log) canonical models of klt pairs. 
\begin{proposition}\label{pro:rational-singularities}
Let $(X, B)$ be a klt pair, where $X$ is a compact analytic variety in Fujiki's class $\mcC$. Assume that the Kodaira dimension $\kappa(X, K_X+B)\>0$. Then $R(X, K_X+B):=\oplus_{m\>0} H^0(X, m(K_X+B))$ is a finitely generated $\mbC$-algebra and \[\bar Z=\Proj R(X, K_X+B)\]
has rational singularities. 
\end{proposition}

\begin{proof}
The finite generation of $R(X, K_X+B)$ follows from \cite[Theorem 1.3]{DHP22} and \cite[Theorem 5.1]{Fuj15}. Let $f:X\bir Z$ be the Iitaka fibration of $K_X+B$. Resolving $Z, f$ and $X$, we may assume that $X$ is a compact K\"ahler manifold, $B$ has SNC support, $Z$ is a smooth projective variety and $f$ is a morphism. Then from the proof of \cite[Theorem 5.1]{Fuj15} it follows that there is a smooth projective variety $Z'$ which is birational to $Z$ and an effective $\mbQ$-divisor $B_{Z'}\>0$ such that $(Z', B_{Z'})$ is klt, $K_{Z'}+B_{Z'}$ is big and the following holds
\[
R(X, K_X+B)^{(d)}\cong R(Z', K_{Z'}+B_{Z'})^{(d')}, 
\]
where the superscripts $d$ and $d'$ represent the corresponding $d$ and $d'$-Veronese subrings.\\
Thus $\bar Z=\Proj R(X, K_X+B)\cong \Proj R(Z', K_{Z'}+B_{Z'})$ is the log-canonical model of $(Z', B_{Z'})$. If $(Z'', B_{Z''})$ is a minimal model of $(Z', B_{Z'})$ as in \cite[Theorem 1.2(2)]{BCHM10}, then by the base-point free theorem, there is a birational morphism $\phi:Z''\to \bar Z$ such that $K_{Z''}+B_{Z''}=\phi^*(K_{\bar Z}+B_{\bar Z})$, where $B_{\bar Z}:=\phi_*B_{Z''}\>0$. Thus $(\bar Z, B_{\bar Z})$ is a klt pair, and hence $\bar Z$ has rational singularities.
\end{proof}

\section{Main Theorem}
In this section we will prove our main theorem. We begin with some preparation.
\begin{definition}\label{def:plurigenera}
Let $X$ be a smooth compact analytic variety. Then the $m$-th \textit{plurigenera} of $X$ is defined as 
\[ 
P_m(X):=\dim_{\mbC} H^0(X, \omega^m_X).
\]
\end{definition}

The next result is one of our main tools in the proof of the main theorem, it is also of independent interest. It follows immediately from the main results of \cite{PPS17}.
 \begin{theorem}\label{thm:maximal-albanese-torus}
Let $X$ be a compact K\"ahler variety with terminal singularities. Assume that $X$ has maximal Albanese dimension and $\kappa (X)=0$. Then $X$ is bimeromorphic to a torus. Additionally, if $K_X$ is also nef, then $X$ is isomorphic to a torus.
\end{theorem}
\begin{remark}
Note that the above result holds if we simply assume that $X$ is in Fujiki's class $\mathcal C$. Indeed, if  $X'\to X$ is a resolution of singularities such that $X'$ is K\"ahler, then $\kappa (X')=0$ and so $X'\to \Alb(X')$ is bimeromorphic, and hence so is $X\to \Alb(X')$. Note also that if $X$ is a complex manifold of maximal Albanese dimension, then $X$ is automatically in Fujiki's class $\mathcal C$. To see this, consider the Stein factorization $X'\to Y\to A$. Then $Y\to A$ is finite and so $Y$ is also K\"ahler (see \cite[Prop. 1.3.1(v) and (vi), page 24]{Var89}). Let $X'\to X$ be a resolution of singularities such that $X'\to Y$ is projective, then $X'$ is K\"ahler and so $X$ is in Fujiki's class $\mathcal C$.
\end{remark}

\begin{proof}[Proof of Theorem \ref{thm:maximal-albanese-torus}]
Since $X$ is terminal, it has rational singularities, and thus by Definition \ref{def:albanese-morphism} the Albanese morphism $a:X\to \Alb(X)$ exists. Let $\pi:\tilde{X}\to X$ be a resolution of singularities of $X$. Then $a\circ\pi:\tilde X\to \Alb(X)$ is the Albanese morphism of $\tilde X$. Moreover, since $X$ has terminal singularities, $\kappa(\tilde X)=\kappa(X)=0$. Thus replacing $X$ by $\tilde X$, we may assume that $X$ is a compact K\"ahler manifold. 
Let $d=\dim X$ and pick a general element $\Theta \in H^0(\Omega _A^d)$, where $A=\Alb(X)$. Then $0\ne a^*\Theta \in H^0(\Omega _X^d)$ and so $P_1(X)>0$.
It follows that $P_k(X)=h^0(X, \omega^k_X)>0$ for all $k>0$.
Since $\kappa (X)=0$, we have $P_1(X)=P_2(X)=1$. Thus by \cite[Theorem 19.1]{PPS17}, $X\to A$ is surjective, and hence $\dim X=\dim A=h^{1,0}(X)$. Thus by \cite[Theorem B]{PPS17}, $X$ is bimeromorphic to a complex torus and so $a:X\to A$ is (surjective and) bimeromorphic.

 Assume now that $X$ has terminal singularities and $K_X$ is nef. Let $a:X\to A$ be the Albanese morphism. By what we have seen above, this morphism is bimeromorphic.  
 
Thus $K_X\num a^*K_A+E\equiv E$, where $E\>0$ is an effective Cartier divisor such that $\Supp(E)=\Ex(a)$ (since $A$ is smooth). By the negativity lemma $E=0$, and hence $a$ is an isomorphism.
 \end{proof}~\\

 \begin{corollary}\label{cor:maximal-albanese-torus}
 Let $(X, B)$ be a compact K\"ahler klt pair. Assume that $X$ has maximal Albanese dimension and $\kappa(X, K_X+B)=0$. Then $X$ is bimeromorphic to a torus. Additionally, if $K_X+B\sim_{\mbQ} 0$, then $X$ is isomorphic to a torus.
 \end{corollary}

 \begin{proof}
 Passing to a terminalization by running an appropriate MMP over $X$ (using \cite[Theorem 1.4]{DHP22}) we may assume that $(X, B)$ has $\mbQ$-factorial terminal singularities. Now since $\kappa(X)\>0$ by Corollary \ref{cor:maximal-albanese-kappa}, $\kappa(X, K_X+B)=0$ implies that $\kappa(X, K_X)=0$. Thus by Theorem \ref{thm:maximal-albanese-torus}, $a:X\to A:=\Alb(X)$ is a surjective bimeromorphic morphism. Now assume that $K_X+B$ is nef. Then $K_X+B=a^*K_A+E+B\sim_{\mbQ} B+E$, where $E\>0$ is an effective Cartier divisor such that $\Supp (E)=\Ex(a)$, since $A$ is smooth. Thus $(B+E)\sim_{\mbQ} 0$, as $K_X+B\sim_{\mbQ} 0$, and hence $B=E=0$ (as $X$ is K\"ahler). In particular, $a:X\to A$ is an isomorphism. 
 \end{proof}~\\

 Now we are ready to prove our main theorem.\\
 
 \begin{proof}[Proof of Theorem \ref{thm:main}]
 Let $a:X\to A$ be the Albanese morphism. Since $X$ has maximal Albanese dimension, $a$ is generically finite over its image $a(X)$. By the relative Chow lemma (see \cite[Corollary 2]{Hir75} and \cite[Theorem 2.12]{DH20}) there is a log resolution $\mu :X'\to X$ of $(X, B)$ such that the Albanese morphism $a'=a\circ \mu :X'\to A$ is projective.  
     Let $K_{X'}+B'=\mu^*(K_X+B)+F$, where $ F\geq 0$ such that $\Supp (F)=\Ex(\mu)$, and $(X', B')$ has klt singularities. Note that if $(X',B')$ has a good minimal model $\psi:X '\dasharrow X^m$, then $\psi$ contracts every component of $F$ and the induced bimeromorphic map $X\dasharrow X^m$ is a good minimal model of $(X,B)$ (see \cite[Lemmas 2.5 and 2.4]{HX13} and their proofs). Thus, we may replace $(X,B)$ by $(X',B')$ and assume that $(X, B)$ is a log smooth pair and $X\to A$ is a projective morphism. 
From Corollary \ref{cor:maximal-albanese-kappa} it follows that $\kappa(X)\>0$. In particular, $\kappa(X, K_X+B)\>0$.    
 Now we split the proof into two parts. In Step 1 we deal with the $\kappa(X, K_X+B)=0$ case, and the remaining cases are dealt with in Step 2.\\
 
 \noindent {\bf Step 1.} Suppose that $\kappa (X, K_X+B)=0$. Then by Theorem \ref{thm:maximal-albanese-torus}, the Albanese morphism $a:X\to A:=\Alb(X)$ is bimeromorphic. Let $D$ be an irreducible component of the unique effective divisor $G\in | m(K_X+B)|$ for $m>0$ sufficiently divisible. We make the following claim.
 
 \begin{claim}\label{clm:a-exceptional}
$D$ is $a$-exceptional; in particular, $G$ is $a$-exceptional.
 \end{claim}

\begin{proof}
First passing to a higher model of $X$ we may assume that $D$ has SNC support. Consider the short exact sequence
 \[0\to \omega _X\to \omega _X(D)\to \omega _D\to 0.\]
 Let $V^0(\omega _D):=\{ P\in {\rm Pic}^0(A)\;|\; h^0(D, \omega _D\otimes a^* P)\ne 0\}$.
If $\dim V^0(\omega _D)> 0$, then it contains a subvariety $K+P$, where $P$ is torsion in ${\rm Pic}^0(A)$ and $K$ is a subtorus of ${\rm Pic}^0(A)$ with $\dim K>0$ (see \cite[Corollary 17.1]{PPS17}).
Since $a:X\to A$ is surjective and bimeromorphic, we have $H^i(X,a^*Q)=H^i(A,Q)=0$ for any $\OO_A\ne Q\in {\rm Pic}^0(A)$; in particular, $H^1(X, \omega_X\otimes a^*Q)=H^{n-1}(X, a^*Q^{-1})^\vee=0$, where $n=\dim X$.
Thus $H^0(X, \omega _X(D)\otimes a^*Q)\to H^0(D, \omega _D\otimes a^*Q)$ is surjective for all $\OO_A\ne Q\in {\rm Pic}^0(A)$, and so $h^0(X, \omega _X(D)\otimes a^*Q)>0$ for all $\OO_A\ne Q\in P+K$. Since $P$ is torsion, $\ell P=0$ for some $\ell>0$.
Consider the morphism
\begin{equation}\label{eqn:non-empty-system}
  |K_X+D+P+Q_1|\times \cdots \times |K_X+D+P+Q_\ell|\to |\ell(K_X+D)|,  
\end{equation}
where $Q_i\in K$ such that $\sum _{i=1}^\ell Q_i=0$.\\
Since $\dim K>0$, for $\ell\geq 2$, the $Q_1,\ldots , Q_\ell$ vary in the subvariety
$\mathcal K\subset K^{\times \ell}$ defined by the equation $\sum _{i=1}^\ell Q_i=0$. Thus $\dim \mathcal K \geq \ell\cdot (\dim K)-1\geq \ell-1\geq 1$. Therefore $\dim |\ell(K_X+D)|>0$, i.e. $ h^0(X, \ell(K_X+D))>1$. Since $D$ is contained in the support of $G$, we have $(r-\ell)G\geq \ell D$ for some $r>0$. Then
$h^0(X, rm(K_X+B))\geq h^0(X, \ell (K_X+D))>1$, which is a contradiction.
Therefore, $\dim V^0(\omega _D)\leq 0$. By \cite[Theorem A]{PPS17},  $a_*\omega _D$ is a GV sheaf so that $\bfR\hat S \mathbf\Delta _A(a_*\omega _D)=\bfR^0\hat S\mathbf\Delta _A(a_*\omega _D)$. If $\dim V^0(\omega _D)= 0$, then $\bfR^0\hat S (\mathbf\Delta _A(a_* \omega _D))$ is an Artinian sheaf of modules on $A$, and hence by Theorem \ref{thm:fm-isomorphism} and 
Remark \ref{rmk:h-vector-bundle}
\[\mathbf\Delta _A(a_*\omega _D)=(-1_A)^*\bfR S(\bfR\hat S \mathbf\Delta _A(a_*\omega _D))[g]=(-1_A)^*\bfR S(\bfR^0\hat S \mathbf\Delta _A (a_*\omega _D))[g]\] is a shift of a homogeneous vector bundle which we denote by $\mcE$ (see Remark \ref{rmk:h-vector-bundle}). But then \[a_*\omega _D=\mathbf\Delta _A(\mathbf\Delta _A(a_*\omega _D))=\mcE^\vee\] is also a homogeneous vector bundle and hence its support is either empty or entire $A$. 
The latter is clearly impossible, since ${\rm Supp}(a_*\omega _D)\ne A$, and hence $V^0(\omega _D)=\emptyset$. Thus by \cite[Proposition 13.6(b)]{PPS17}, $a_*\omega _D=0$; in particular $D$ is $a$-exceptional.
 
 \end{proof}~\\

Now by \cite[Theorem 1.4]{DHP22} and \cite[Theorem 1.1]{Fuj22} we can run the relative minimal model program over $A$ and hence may assume that $K_X+B$ is nef over $A$. From our claim above we know that $K_X+B\sim _{\mathbb Q}E\geq 0$ for some effective $a$-exceptional divisor $E\>0$. Then by the negativity lemma we have $E=0$; thus $\mcO_X(m(K_X+B))\cong\mcO_X$ for sufficiently divisible $m>0$, and hence we have a good minimal model.\\

\noindent {\bf Step 2.} Suppose now that  $\kappa (X, K_X+B)\geq 1$ and let $f:X\bir Z$ be the Iitaka fibration. Note that the ring $R(X, K_X+B):=\oplus_{m\>0}H^0(X, \mcO_X(\lrd m(K_X+B)\rrd))$ is a finitely generated $\mbC$-algebra by \cite[Theorem 1.3]{DHP22}. Define $\bar Z:=\Proj R(X, K_X+B)$. Then $Z\bir \bar Z$ is a birational map of projective varieties. Resolving the graph of $Z\bir \bar Z$ we may assume that $Z$ is a smooth projective variety and $\nu:Z\to \bar Z$ is a birational morphism. Then passing to a resolution of $X$ we may assume that $f$ is a morphism and $(X, B)$ is a log smooth pair. Write $K_F+B_F=(K_X+B)|_F$, where $F$ is a very general fiber of $f$, so that $\kappa (F, K_F+B_F)=0$. Note that $a|_F$ is also generically finite (as $F$ is a very general fiber of $f$) and thus
$F$ has maximal Albanese dimension. In particular, $(F,B_F)$ has a good minimal model by {Step 1}.
Let $\psi : F\bir F'$ be this minimal model; then $K_{F'}+B_{F'}\sim_{\mbQ} 0$. Thus by Corollary \ref{cor:maximal-albanese-torus}, $F'$ is a torus and $B_{F'}=0$; in particular, $\psi:F\to F'$ is the Albanese morphism. Thus $a|_F:F\to A$ factors through $\psi:F\to F'$; let $\alpha:F'\to A$ be the induced morphism. Let $K:=\alpha(F')$; then $K$ is a torus, and $\alpha$ is \'etale over $K$, as $F'$ and $K$ are both homogeneous varieties. Now since $A$ contains at most countably many subtori and $F$ is a very general fiber, $K$ is independent of the very general points $z\in Z$, and hence so is $F'$. Define $A':=A/K$, then $A'$ is again a torus. Since the composite morphism $X\to A'$ contracts $F$ and $\dim F=\dim K$, from the rigidity lemma (see \cite[Lemma 4.1.13]{BS95}) and dimension count it follows that there is a meromorphic map $Z\bir A'$ generically finite onto its image. Since $Z$ is smooth, we may assume that $Z\to A'$ is a morphism (see \cite[Lemma 8.1]{Kaw85}). Similarly, since $\bar Z$ has rational singularities by Proposition \ref{pro:rational-singularities}, again from \cite[Lemma 8.1]{Kaw85} it follows that $\bar Z\to A'$ is a morphism.\\

Now choose an ample $\mbQ$-divisor $\bar H$ on $\bar Z$. Let $H_X$ be the pull-back of $\bar H$ to $X$ such that $K_X+B\sim _{\mathbb Q} H_X+E$ and $|k(K_X+B)|=|k H_X|+kE$ for any sufficiently large and divisible integer $k>0$, where $E\>0$ is effective. 

Now let $\bar A :=\bar Z \times _{A'}A$. Observe that there is a unique morphism $\bar a:X\to \bar A$ determined by the universal property of fiber products. 
We claim that $E$ is exceptional over $\bar A$.  If not, then let $D$ be a component of $E$ which is not exceptional over $\bar A$. Let $h:X\to \bar Z$ be the composite morphism $X\to Z\to \bar Z$ and $W:=h(D)$.  
Choose a sufficiently divisible and large positive integer $s>0$ such that $s\bar H$ is very ample, $r(K_X+B)$ is Cartier, $rE\geq D$ and $|r(K_X+B)|=|rH_X|+rE$, where $r=(n+1)s$ and $n=\dim X$. 

\begin{equation}
\xymatrixcolsep{3pc}\xymatrixrowsep{3pc}\xymatrix{
X\ar@/_1pc/[dd]^f\ar[dr]_{\bar a}\ar@/^2pc/[rrd]^a \\
& \bar A:=\bar Z\times_{A'} A\ar[d]\ar[r] & A\ar[d]\\
Z\ar[r] & \bar Z\ar[r] & A':=A/K
}
\end{equation}

\begin{claim}\label{clm:non-vanishing}
$|K_D+(n+1)sH_D|\ne \emptyset$, where $H_D=H_X|_{D}$.
\end{claim}
\begin{proof} Let $D_i=G_1\cap \ldots \cap G_i$ be the intersection of general divisors
$G_1,\ldots , G_m\in |sH_D|$, where $0\leq i\leq m:=\dim W$ and $D_0:=D$.
Let $M:=K_D+(n+1)sH_D$, then we have the short exact sequences
\[0\to \mathcal O _{D_i}(M-G_{i+1}) \to \mathcal O _{D_i}(M)\to \mathcal O _{D_{i+1}}(M)\to 0.\]
Recall that $h:X\to \bar Z$ is the given morphism; let $h_i:=h|_{D_i}$. Then
\begin{align*}
   (M-G_{i+1})|_{D_i} &\sim (K_D+nsH_D)|_{D_i}\\
                      &\sim \Bigl(K_D+\sum_{j=1}^i G_j+(n-i)sH_D\Bigr)\Big|_{D_i}\\
                      &\sim K_{D_i}+(n-i)sH_{D_i}\\
                      &\sim K_{D_i}+h_i^* (n-i)s\bar H,
\end{align*}
 where $H_{D_i}:=H_X|_{D_i}$. 
By \cite[Theorem 3.1(i)]{Fuj22b} the only associated subvarieties of 
\[R^1h_{i,*}\mcO_{D_i}(M-G_{i+1})=R^1h_{i,*} \mcO_{D_i}(K_{D_i})\otimes \mcO _{\bar Z}((n-i)s\bar H)\] are 
$W_i:=h(D_i)\subset \bar Z$, i.e. $R^1h_{i,*}\mcO_{D_i}(M-G_{i+1})$ is a torsion free sheaf on $W_i$.
 
Therefore, the induced homomorphism $h_{i,*}\mcO_{D_{i+1}}(M)\to R^1h_{i,*}\mcO_{D_i}(M-G_{i+1})$ is zero and we have the following exact sequence
\[
\xymatrixcolsep{3pc}\xymatrix{
0\ar[r] & h_{i,*}\mcO_{D_i}(M-G_{i+1})\ar[r] & h_{i,*}\mcO_{D_i}(M)\ar[r] & h_{i,*}\mcO_{D_{i+1}}(M)\ar[r] & 0.
}
\]
By \cite[Theorem 3.1(ii)]{Fuj22b} we have
\[H^1(h_{i,*}\mathcal O _{D_i}(M-G_{i+1}))=H^1(h_{i,*} \mcO_{D_i}(K_{D_i})\otimes \mcO _{\bar Z}((n-i)s\bar H))=0,\] 
and thus we have the following surjections
\begin{equation}\label{eqn:surjections}
     H^0( \OO _{D}(M))\to H^0( \OO _{D_1}(M_{D_1}))\to \cdots \to H^0( \OO _{D_m}(M_{D_m}))\to H^0(\OO _G(M|_{G})),
\end{equation}
where $G$ is a connected (and hence irreducible, as $D_m$ is smooth) component of $D_m$. Note that $G$ is a general fiber of $D\to W$, since $H_{D}$ is a pullback from $W$ and $m=\dim W$.\\

Let $w:=h(G)\in W\subset \bar Z$. Then 
$G\to \bar G:=\bar a(G)$ is generically finite (as so is $D\to \bar a (D)$ by our assumption), and 
$\bar G \to a(G)$ is an isomorphism, since $\bar A _w\to K\subset A$ is an isomorphism, as 
$\bar A _w=(A\times _{A'}\bar Z)_w=A\times _{A'}\{w\}\cong K$. In particular, $G$ has maximal Albanese dimension, and hence $h^0(G, K_G)>0$ by Lemma \ref{lem:effective-k-dim}. Now since $M|_{G}\sim K_G$, from the surjections in \eqref{eqn:surjections} it follows that  $|M|=|K_D+(n+1)sH_D|\ne \emptyset$, and hence the claim follows.

\end{proof}~\\

Now consider the short exact sequence
\[0\to \omega _X(L)\to \omega _X(L+D)\to \omega _D(L)\to 0, \] 
where $L= rH_X$.
Then by \cite[Theorem 3.1(i)]{Fuj22b}, $R^1h_* \omega _X(L)=R^1h_* \omega _X\otimes \OO _{\bar Z}(r\bar H)$ is torsion free, and hence 
$h_* \omega _X(L+D)\to h_* \omega _D(L)$ is surjective. Again by \cite[Theorem 3.1(ii)]{Fuj22b}, $H^1(h_* \omega _X(L))=H^1(h_* \omega _X\otimes \OO _{\bar Z}(r\bar H))=0$, and so $H^0(\omega _X(L+D))\to H^0(\omega _D(L))$ is surjective.
Since $|K_D+L|_D|\ne 0$ by Claim \ref{clm:non-vanishing}, $D$ is not contained in the base locus of $|K_X+L+D|$.
Let $0\leq b:={\rm mult}_D(B)<1$ and $e:={\rm mult }_D(E)>0$. Then $\sigma E+B-D\geq 0$ and ${\rm mult}_D(\sigma E+B-D)=0$ for $\sigma =\frac {1-b} e>0$. We may assume that $\sigma \leq r$ (as $r$ is sufficiently large and divisible). Adding $rE+B-D$ to a general divisor $G\in|K_X+L+D|$ we get
\[\Gamma:=rE+B-D+G\sim_{\mbQ}  (r+1)(K_X+B)\sim_\Q (r+1)(H_X+E).\] Then for any sufficiently divisible $m>0$ we have 
\[{\mult}_D(m\Gamma ) = m(r-\sigma ){\mult}_D(E)<m(r+1){\mult}_D(E),\] which is a contradiction to the fact that $|k(K_X+B)|=|kH_X|+kE$ for sufficiently divisible $k=m(r+1)>0$. Thus $D$ is exceptional over $\bar A$.\\

Let $n=\dim X$. We will run a relative $(K_{X}+B+(2n+3)s H_{X})$-MMP over $A$. Note that since $|(2n+3)sH_X|$ is a base-point free linear system on a smooth compact variety $X$, by Sard's theorem there is an effective $\mbQ$-divisor $H'\>0$ such that $(2n+3)sH_X\sim_{\mbQ} H'$ and $(X, B+H')$ has klt singularities. Thus $K_X+B+(2n+3)sH_X\sim_{\mbQ} K_X+B+H'$ and we can run a $(K_X+B+(2n+3)sH_X)$-MMP over $A$ by \cite[Theorem 1.4]{DHP22}, and obtain $X\dasharrow X'$ so that $K_{X'}+B'+(2n+3)s H_{X'}\sim_{\mbQ} ((2n+3)s+1)H_{X'}+E'$ is nef over $A$. Note that if $R$ is a $(K_{X}+B+(2n+3)s H_{X})$-negative extremal ray over $A$, then it is also $(K_{X}+B)$-negative and so it is spanned by a rational curve $C$ such that $0>(K_{X}+B)\cdot C\geq -2n$ (see \cite[Theorem 2.46]{DHP22}). But then $C$ is vertical over $\bar Z$, otherwise $(K_{X}+B+(2n+3)s H_{X})\cdot C>0$, as $H_X$ is the  pullback of an ample divisor $\bar H$ on $\bar Z$, this is a contradiction.
Thus it follows that every step of this MMP is also a step of an MMP over $\bar Z$, and hence there is an induced morphism $\mu:X'\to \bar A:=\bar Z \times _{A'}A$.
It follows that \[K_{X'}+B'\sim _\Q \mu ^* H_{\bar A}+E'\sim _{\Q, \bar A}E'\geq 0,\]
where $H_{\bar A}$ is the pullback of the ample divisor $\bar H$ by the projection $\bar A\to \bar Z$.\\
Then $E'$ is nef and exceptional over $\bar A$, and hence by the negativity lemma, $E'=0$.
But then $K_{X'}+B'\sim _\Q \mu ^* H_{\bar A}$ and since $H_{\bar A}$ is semi-ample, so is $K_{X'}+B'$.

 \end{proof}

\begin{corollary}
Let $(X,B)$ be a compact K\"ahler klt pair of maximal Albanese dimension such that $a:X\to A:=\Alb(X)$ is a projective morphism. Then  we can run a $(K_X+B)$-Minimal Model Program which ends with a good minimal model.
\end{corollary}
\begin{proof}
Note that since $a:X\to A$ is generically finite over image, $K_X+B$ is relatively big over $a(X)$.  Thus by \cite[Theorem 1.4]{DHP22} and \cite[Theorem 1.8]{Fuj22}, we can run a $(K_X+B)$-Minimal Model Program over $A$. Notice that each step of this MMP is also a step of the $(K_X+B)$-MMP. Therefore, we may assume that $K_X+B$ is nef over $A$ and we must check that it is indeed nef on $X$.
Let $(\bar X, \bar B)$ be a good minimal model of $(X,B)$, which exists by Theorem \ref{thm:main}. By what we have seen, $(\bar X, \bar B)$ is also a minimal model over $A$. But then 
$\phi:(X,B)\dasharrow (\bar X, \bar B)$ is an isomorphism in codimension $1$. If $p:Y\to X$ and $q:Y\to \bar X$ is a common resolution, then from the negativity lemma it follows that $p^*(K_X+B)=q^*(K_{\bar X}+\bar B)$. In particular, $p^*(K_X+B)$ is semi-ample, and hence so is $K_X+B$. Thus $(X,B)$ is a good minimal model.
\end{proof}

\bibliographystyle{hep}
\bibliography{4foldreferences}

\begin{thebibliography}{BCHM10}

\bibitem[BCHM10]{BCHM10}
C.~Birkar, P.~Cascini, C.~D. Hacon and J.~McKernan, \textsl{ Existence of
  minimal models for varieties of log general type},
\newblock J. Amer. Math. Soc. \textbf{ 23}(2), 405--468 (2010).

\bibitem[BS95]{BS95}
M.~C. Beltrametti and A.~J. Sommese,
\newblock \textsl{ The adjunction theory of complex projective varieties},
  volume~16 of \textsl{ De Gruyter Expositions in Mathematics},
\newblock Walter de Gruyter \& Co., Berlin, 1995.

\bibitem[CH20]{CH20}
J.~Cao and A.~H\"{o}ring, \textsl{ Rational curves on compact {K}\"{a}hler
  manifolds},
\newblock J. Differential Geom. \textbf{ 114}(1), 1--39 (2020).

\bibitem[DH20]{DH20}
O.~{Das} and C.~{Hacon}, \textsl{ {The log minimal model program for K{\"a}hler
  $3$-folds}},
\newblock arXiv e-prints , arXiv:2009.05924 (September 2020), {2009.05924}.

\bibitem[DHP22]{DHP22}
O.~{Das}, C.~{Hacon} and M.~{P{\u{a}}un}, \textsl{ {On the $4$-dimensional
  minimal model program for K{\"a}hler varieties}},
\newblock arXiv e-prints , arXiv:2205.12205 (May 2022), {2205.12205}.

\bibitem[Fuj15]{Fuj15}
O.~Fujino, \textsl{ Some Remarks on the Minimal Model Program for Log Canonical
  Pairs},
\newblock J. Math. Sci. Univ. Tokyo \textbf{ 22}, 149--192 (2015).

\bibitem[Fuj22a]{Fuj22}
O.~Fujino, \textsl{ Minimal model program for projective morphisms between
  complex analytic spaces},
\newblock (2022).

\bibitem[{Fuj}22b]{Fuj22b}
O.~{Fujino}, \textsl{ {Vanishing theorems for projective morphisms between
  complex analytic spaces}},
\newblock arXiv e-prints , arXiv:2205.14801 (May 2022), {2205.14801}.

\bibitem[Hir75]{Hir75}
H.~Hironaka, \textsl{ Flattening theorem in complex-analytic geometry},
\newblock Amer. J. Math. \textbf{ 97}, 503--547 (1975).

\bibitem[HP16]{HP16}
A.~H\"oring and T.~Peternell, \textsl{ Minimal models for {K}\"ahler
  threefolds},
\newblock Invent. Math. \textbf{ 203}(1), 217--264 (2016).

\bibitem[HX13]{HX13}
C.~D. Hacon and C.~Xu, \textsl{ Existence of log canonical closures},
\newblock Invent. Math. \textbf{ 192}(1), 161--195 (2013).

\bibitem[Kaw85]{Kaw85}
Y.~Kawamata, \textsl{ Minimal models and the {K}odaira dimension of algebraic
  fiber spaces},
\newblock J. Reine Angew. Math. \textbf{ 363}, 1--46 (1985).

\bibitem[Muk81]{Mu81}
S.~Mukai, \textsl{ Duality between $D(X)$ and $D(\hat X)$ with its application
  to Picard sheaves},
\newblock Nagoya math. J. \textbf{ 81}, 153--175 (1981).

\bibitem[PPS17]{PPS17}
G.~Pareschi, M.~Popa and C.~Schnell, \textsl{ Hodge modules on complex tori and
  generic vanishing for compact {K}\"{a}hler manifolds},
\newblock Geom. Topol. \textbf{ 21}(4), 2419--2460 (2017).

\bibitem[Var89]{Var89}
J.~Varouchas, \textsl{ K\"{a}hler spaces and proper open morphisms},
\newblock Math. Ann. \textbf{ 283}(1), 13--52 (1989).

\end{thebibliography}
\end{document}